\documentclass[11pt,reqno]{amsart}

\newtheorem{proposition}{Proposition}[section]

\newtheorem{corollary}[proposition]{Corollary}
\newtheorem{theorem}[proposition]{Theorem}

\theoremstyle{definition}
\newtheorem{definition}[proposition]{Definition}
\newtheorem{example}[proposition]{Example}
\newtheorem{examples}[proposition]{Examples}

\newcommand{\thlabel}[1]{\label{th:#1}}
\newcommand{\thref}[1]{Theorem~\ref{th:#1}}
\newcommand{\selabel}[1]{\label{se:#1}}
\newcommand{\seref}[1]{Section~\ref{se:#1}}

\newcommand{\prlabel}[1]{\label{pr:#1}}
\newcommand{\prref}[1]{Proposition~\ref{pr:#1}}
\newcommand{\colabel}[1]{\label{co:#1}}
\newcommand{\coref}[1]{Corollary~\ref{co:#1}}

\newcommand{\exlabel}[1]{\label{ex:#1}}
\newcommand{\exref}[1]{Example~\ref{ex:#1}}

\newcommand{\eqlabel}[1]{\label{eq:#1}}
\newcommand{\equref}[1]{(\ref{eq:#1})}

\def\ot{\otimes}

\newcommand{\Cc}{\mathcal{C}}

\def\*C{{}^*\hspace*{-1pt}{\Cc}}
\def\text#1{{\rm {\rm #1}}}

\input xy
\xyoption {all} \CompileMatrices

\usepackage{color,amssymb,graphicx}

\begin{document}

\title[Unified products and split extensions of Hopf algebras]
{Unified products and split extensions of Hopf algebras}

\author{A. L. Agore}\thanks{A.L. Agore is research fellow ``aspirant'' of FWO-Vlaanderen; G. Militaru
was supported by the CNCS - UEFISCDI grant no. 88/5.10.2011 'Hopf
algebras and related topics'.}
\address{Faculty of Engineering, Vrije Universiteit Brussel, Pleinlaan 2, B-1050 Brussels, Belgium}
\email{ana.agore@vub.ac.be and ana.agore@gmail.com}

\author{G. Militaru}
\address{Faculty of Mathematics and Computer Science, University of Bucharest, Str.
Academiei 14, RO-010014 Bucharest 1, Romania}
\email{gigel.militaru@fmi.unibuc.ro and gigel.militaru@gmail.com}


\begin{abstract}
The unified product was defined in \cite{am3} related to the
restricted extending structure problem for Hopf algebras: a Hopf
algebra $E$ factorizes through a Hopf subalgebra $A$ and a
subcoalgebra $H$ such that $1\in H$ if and only if $E$ is
isomorphic to a unified product $A \ltimes H$. Using the concept
of normality of a morphism of coalgebras in the sense of
Andruskiewitsch and Devoto we prove an equivalent description for
the unified product from the point of view of split morphisms of
Hopf algebras. A Hopf algebra $E$ is isomorphic to a unified
product $A \ltimes H$ if and only if there exists a morphism of
Hopf algebras $i: A \rightarrow E$ which has a retraction $\pi: E
\to A$ that is a normal left $A$-module coalgebra morphism. A
necessary and sufficient condition for the canonical morphism $i :
A \to A\ltimes H$ to be a split monomorphism of bialgebras is
proved, i.e. a condition for the unified product $A\ltimes H$ to
be isomorphic to a Radford biproduct $L \ast A$, for some
bialgebra $L$ in the category $_{A}^{A}{\mathcal YD}$ of
Yetter-Drinfel'd modules. As a consequence, we present a general
method to construct unified products arising from an unitary not
necessarily associative bialgebra $H$ that is a right $A$-module
coalgebra and a unitary coalgebra map $\gamma : H \to A$
satisfying four compatibility conditions. Such an example is
worked out in detail for a group $G$, a pointed right $G$-set $(X,
\cdot, \lhd )$ and a map $\gamma : G \to X$.
\end{abstract}

\maketitle

\section*{Introduction}
A morphism $i : C \to D$ in a category ${\mathcal C}$ is a split
monomorphism if there exists $p: D \to C$ a morphism in ${\mathcal
C}$ such that $p\circ i = {\rm Id}_C$. Fundamental constructions
like the semidirect product of groups or Lie algebras, the smash
product of Hopf algebras, Radford's biproduct, etc can be viewed
as tools to answer the following problem: \emph{describe, when is
possible, split monomorphisms in a given category ${\mathcal C}$.}
The basic example is the following: a group $E$ is isomorphic to a
semidirect product $G\ltimes A$ of groups if and only if there
exists $i : A \to E$ a split monomorphism of groups. In this case
$E \cong G\ltimes A$, where $G = {\rm Ker} (p)$, for a splitting
morphism $p: E \to A$. The generalization of this elementary
result at the level of Hopf algebras was done in two steps. The
first one was made by Molnar in \cite[Theorem 4.1]{molnar}: if $i
: A \to E$ is a split monomorphism of Hopf algebras having a
splitting map $p: E \to A$ which is a \emph{normal} morphism of
Hopf algebras then $E$ is isomorphic as a Hopf algebra to a smash
product $G \# A$ of Hopf algebras, where $G := {\rm KER} (p) = \{
x\in E \, | \, x_{(1)} \ot p (x_{(2)}) \ot x_{(3)} = x_{(1)} \ot
1_A \ot x_{(2)} \}$, is the kernel of $p$ in the category of Hopf
algebras (the original statement of Molnar's theorem, as well as
the one of Radford below, is restated and updated according to the
present development of the theory). The general case was done by
Radford in \cite[Theorem 3]{radford}: if $i : A \to E$ is a split
monomorphism of Hopf algebras then $E$ is isomorphic as a Hopf
algebra to a biproduct $G \ast A$, for a bialgebra $G$ in the
pre-braided category $_{A}^{A}{\mathcal YD}$ of left-left
Yetter-Drinfel'd modules. The price paid for leaving aside the
normality assumption of the splitting map is reasonable: the
algebra structure of the biproduct $G \ast A$ is the one of the
smash product algebras and the coalgebra structure of $G \ast A$
is given by the smash coproduct of coalgebras. In the last decade,
the biproduct (also named bosonization) was intensively used in
the classification of finite dimensional pointed Hopf algebras
(see \cite{AS} and the references therein).

The great impact generated by Radford's theorem made his result
the subject of many generalizations and recent developments
(\cite{ABM}, \cite{AM}, \cite{AMS}, \cite{Sch2}). The first step
was made by Schauenburg, who considered the splitting map $p$ to
be only a coalgebra map. Although it is formulated in a different
manner, \cite[Theorem 5.1]{Sch2} proves that if a morphism of Hopf
algebras $i : A \to E$ has a retraction $\pi: E \to A$ which is a
left $A$-module coalgebra morphism then $E$ is isomorphic to a new
product $G \propto A$, that generalizes the biproduct, in the
construction of which $G$ and $A$ are connected by four maps
satisfying seventeen compatibility conditions. As a vector space
$G \propto A = G \ot A$ with the multiplication given by a very
laborious formula. The special case in which the splitting map $p$
is a coalgebra $A$-bimodule map was highlighted briefly in
\cite[Section 6.1]{Sch2} and in full detail in \cite[Theorem
3.64]{AMSt}.

The unified product was constructed in \cite{am3} from a
completely different reason, related to what we have called the
\emph{extending structures problem}. Given a Hopf algebra $A$ and
a coalgebra $E$ such that $A$ is a subcoalgebra of $E$, the
extending structures problem asks for the description and
classification of all Hopf algebra structures that can be defined
on the coalgebra $E$ such that $A$ becomes a Hopf subalgebra of
$E$. The unified product is the tool for solving the restricted
version of the ES-problem: a Hopf algebra $E$ factorizes through a
Hopf subalgebra $A$ and a subcoalgebra $H$ such that $1\in H$ if
and only if $E$ is isomorphic to a unified product $A \ltimes H$.
In the construction of a unified product $A \ltimes H$ the Hopf
algebra $A$ and the coalgebra $H$ are connected by three coalgebra
maps: two actions $\triangleleft : H \otimes A \rightarrow H$,
$\triangleright: H \otimes A \rightarrow A$ and a generalized
cocycle $f: H \otimes H \rightarrow A$ satisfying seven
compatibility conditions.

In this paper we shall prove an equivalent description for the
unified product from the view point of split extensions of Hopf
algebras. Let $A \ltimes H$ be a unified product associated to a
bialgebra extending structure $\Omega(A) = \bigl(H, \triangleleft,
\, \triangleright, \, f \bigl)$ of a Hopf algebra $A$ (see
\seref{prel}). Then we have an extension of bialgebras $i_A: A \to
A \ltimes H$. This extension is split in the sense of \cite{Sch2}:
there exists $\pi_A : A \ltimes H \to A$ a left $A$-module
coalgebra morphism such that $\pi_A \circ i_A = {\rm Id}_A$. Thus
the unified product $A \ltimes H$ appears as a special case of the
Schauenburg's product $A \propto H$ but is a much more malleable
version of it. Furthermore, there is more to be said and this
makes a substantial difference: $\pi_A$ is also a \emph{normal}
morphism of coalgebras in the sense of Andruskiewitsch and Devoto
\cite{AD}. This context fully characterizes unified products: we
prove that a Hopf algebra $E$ is isomorphic to a unified product
$A \ltimes H$ if and only if there exists a morphism of Hopf
algebras $i: A \rightarrow E$ which has a retraction $\pi: E \to
A$ that is a normal left $A$-module coalgebra morphism
(\thref{3split}). The proof of \thref{3split} is different from
the one in \cite{Sch2} and is based on the factorization theorem
\cite[Theorem 2.7]{am3}. As a bonus, a simplified and more
transparent version of \cite[Theorem 4.1]{molnar} is obtained in
\coref{molnar}: if $\pi : E \to A$ is a normal split epimorphism
of Hopf algebras, then $ E \cong A \# H$, where $A \# H$ is the
right version of the smash product of bialgebras for some right
$A$-module bialgebra $H$.

The next aim of the paper is to make the connection between the
unified product and Radford's biproduct: for a Hopf algebra $A$,
\prref{splitmono1} gives necessary and sufficient conditions for
$i_A: A \to A \ltimes H$ to be a split monomorphism of bialgebras.
In this case the unified product $A \ltimes H$ is isomorphic as a
bialgebra to a biproduct $L \ast A$, and the structure of $L$ as a
bialgebra in the category ${}_{A}^{A}{}{\mathcal YD}$ of
Yetter-Drinfel'd modules is explicitly described. Furthermore, in
this context a new equivalent description of the unified product
$A\ltimes H$, as well as of the associated biproduct $L \ast A$,
is given in \prref{produsulciudat}: both of them are isomorphic to
a new product $A \circledast H$, which is a deformation of the
smash product of bialgebras $A\# H$ using a coalgebra map $\gamma
: H \to A$. Finally, \thref{construnifbirp} gives a general method
for constructing unified products, as well as biproducts arising
from a right $A$-module coalgebra $(H, \lhd)$ and a unitary
coalgebra map $\gamma : H \to A$. \exref{ultex} gives an explicit
example of such a product starting with a group $G$, a pointed
right $G$-set $(X, \cdot, \lhd )$ and a map $\gamma : G \to X$
satisfying two compatibility conditions.

\section{Preliminaries}\selabel{prel}
Throughout this paper, $k$ will be a field. Unless specified
otherwise, all modules, algebras, coalgebras, bialgebras, tensor
products, homomorphisms and so on are over $k$. For a coalgebra
$C$, we use Sweedler's $\Sigma$-notation: $\Delta(c) = c_{(1)}\ot
c_{(2)}$, $(I\ot\Delta)\Delta(c) = c_{(1)}\ot c_{(2)}\ot c_{(3)}$,
etc (summation understood). We also use the Sweedler notation for
left $C$-comodules: $\rho (m) = m_{<-1>}\otimes m_{<0>}$, for any
$m \in M$ if $(M,\rho) \in {}^C{\mathcal M}$ is a left
$C$-comodule. Let $A$ be a bialgebra and $H$ an algebra and a
coalgebra. A $k$-linear map $f: H \ot H \to A$ will be denoted by
$f (g, \, h) = f (g\ot h)$; $f$ is the \textit{trivial map} if $f
(g, h) = \varepsilon_H (g) \varepsilon_H (h) 1_A$, for all $g$,
$h\in H$. Similarly, the $k$-linear maps $\triangleleft : H
\otimes A \rightarrow H$, $\triangleright: H \otimes A \rightarrow
A$ are the \textit{trivial actions} if $h \triangleleft a =
\varepsilon_A (a) h$ and respectively $h\triangleright a =
\varepsilon_H(h) a$, for all $a\in A$ and $h\in H$.

For a Hopf algebra $A$ we denote by
${}_{A}^A{\mathcal M}$ the category of left-left $A$-Hopf modules:
the objects are triples $(M, \cdot, \rho)$, where $(M, \cdot) \in
{}_A{\mathcal M}$ is a left $A$-module, $(M, \rho) \in
{}^A{\mathcal M}$ is a left $A$-comodule such that
$$
\rho (a \cdot m) = a_{(1)} m_{<-1>} \otimes a_{(2)} \cdot m_{<0>}
$$
for all $a\in A$ and $m\in M$. For $ (M, \cdot, \rho) \in
{}_{A}^A{\mathcal M}$ we denote by $M^{{\rm co}(A)} = \{ m\in M \,
| \, \rho (m) = 1_A \ot m \}$ the subspace of coinvariants.  The
fundamental theorem for Hopf modules states that for any $A$-Hopf
module $M$ the canonical map
$$
\varphi : A \ot M^{{\rm co}(A)}  \to M, \quad \varphi (a \ot m) :=
a\cdot m
$$
for all $a\in A$ and $m\in M$ is bijective with the inverse given
by
$$
\varphi^{-1} : M \to A \ot M^{{\rm co}(A)}, \quad \varphi^{-1} (m)
:= m_{<-2>} \ot S (m_{<-1>}) \cdot m_{<0>}
$$
for all $m\in M$.

${}_{A}^{A}{}{\mathcal YD}$ will denote the pre-braided monoidal
category of left-left Yetter-Drinfel'd modules over $A$: the
objects are triples $(M, \cdot, \rho)$, where $(M, \cdot) \in
{}_A{\mathcal M}$ is a left $A$-module, $(M, \rho) \in
{}^A{\mathcal M}$ is a left $A$-comodule such that
$$
\rho (a \cdot m) = a_{(1)} m_{<-1>} S(a_{(3)}) \otimes a_{(2)} \cdot m_{<0>}
$$
for all $a\in A$ and $m\in M$. If $(L, \cdot, \rho)$ is a
bialgebra in the monoidal category ${}_{A}^{A}{}{\mathcal YD}$
then the Radford biproduct $L \ast A$ is the vector space $L \ot
A$ with the bialgebra structure given by
\begin{eqnarray*}
(l \ast a) \, (m \ast b) &:=& l ( a_{(1)}\cdot m) \ast a_{(2)} \, b\\
\Delta (l \ast a) &:=& l_{(1)} \ast l_{(2)<-1>} a_{(1)} \ot l_{(2)<0>} \ast a_{(2)}
\end{eqnarray*}
for all $l$, $m\in L$ and $a$, $b\in A$, where we denoted $l \otimes
a \in L \otimes A$ by $l \ast a$. $L \ast A$ is a bialgebra with
the unit $1_{L} \ast 1_{A}$ and the counit $\varepsilon_{L \ast A} (l \ast a) =
\varepsilon_{L}(l)\varepsilon_{A}(a)$, for all $l\in L$ and $a\in A$.

We recall from \cite{am3} the construction of the unified product.
An \textit{extending datum} of a bialgebra $A$ is a system
$\Omega(A) = \bigl(H, \triangleleft, \, \triangleright, \, f
\bigl)$, where $H = \bigl( H, \Delta_{H}, \varepsilon_{H}, 1_{H},
\cdot \bigl)$ is a $k$-module such that $\bigl( H, \Delta_{H},
\varepsilon_{H}\bigl)$ is a coalgebra, $\bigl( H, 1_{H}, \cdot
\bigl)$ is an unitary not necessarily associative $k$-algebra, the
$k$-linear maps $\triangleleft : H \otimes A \rightarrow H$,
$\triangleright: H \otimes A \rightarrow A$, $f: H \otimes H
\rightarrow A$ are coalgebra maps such that the following
normalization conditions hold:
\begin{equation}\eqlabel{2}
\quad h \triangleright 1_{A} = \varepsilon_{H}(h)1_{A}, \quad
1_{H} \triangleright a = a, \quad 1_{H} \triangleleft a =
\varepsilon_{A}(a)1_{H}, \quad h\triangleleft 1_{A} = h
\end{equation}
\begin{equation}\eqlabel{3}
\Delta_{H}(1_{H}) = 1_{H} \otimes 1_{H}, \qquad f(h, 1_{H}) =
f(1_{H}, h) = \varepsilon_{H}(h)1_{A}
\end{equation}
for all $h \in H$, $a \in A$.

Let $\Omega(A) = \bigl(H, \triangleleft, \, \triangleright, \, f
\bigl)$ be an extending datum of $A$. We denote by $A
\ltimes_{\Omega(A)} H = A \ltimes H$ the $k$-module $A \otimes H$
together with the multiplication:
\begin{equation}\eqlabel{10}
(a \ltimes h)\bullet(c \ltimes g) := a(h_{(1)}\triangleright
c_{(1)})f\bigl(h_{(2)}\triangleleft c_{(2)}, \, g_{(1)}\bigl) \,
\ltimes \, (h_{(3)}\triangleleft c_{(3)}) \cdot g_{(2)}
\end{equation}
for all $a, c \in A$ and $h, g \in H$, where we denoted $a \otimes
h \in A \otimes H$ by $a \ltimes h$. The object $A \ltimes H$ is
called \textit{the unified product of $A$ and $\Omega(A)$} if $A
\ltimes H$ is a bialgebra with the multiplication given by
\equref{10}, the unit $1_{A} \ltimes 1_{H}$ and the coalgebra
structure given by the tensor product of coalgebras, i.e.:
\begin{eqnarray}
\Delta_{A \ltimes H} (a \ltimes h) &{=}& a_{(1)} \ltimes h_{(1)}
\otimes a_{(2)}
\ltimes h_{(2)}\eqlabel{11}\\
\varepsilon_{A \ltimes H} (a \ltimes h) &{=}&
\varepsilon_{A}(a)\varepsilon_{H}(h)\eqlabel{12}
\end{eqnarray}
for all $h \in H$, $a \in A$. We have proved in \cite[Theorem
2.4]{am3} that $A \ltimes H$ is an unified product if and only if
$\Delta_{H} : H \to H\otimes H$ and $\varepsilon_{H} : H \to k$
are $k$-algebra maps, $(H, \lhd)$ is a right $A$-module structure
and the following compatibilities hold:
\begin{enumerate}
\item[(BE1)] $(g\cdot h)\cdot l = \bigl(g \triangleleft
f(h_{(1)}, \, l_{(1)})\bigl)\cdot (h_{(2)}\cdot l_{(2)})$\\
\item[(BE2)] $g \triangleright (ab) = (g_{(1)} \triangleright
a_{(1)})[(g_{(2)}\triangleleft a_{(2)})\triangleright b]$\\
\item[(BE3)] $(g\cdot h) \triangleleft a = [g \triangleleft
(h_{(1)}
\triangleright a_{(1)})] \cdot (h_{(2)} \triangleleft a_{(2)})$\\
\item[(BE4)] $[g_{(1)} \triangleright (h_{(1)} \triangleright
a_{(1)})]f\Bigl(g_{(2)} \triangleleft (h_{(2)} \triangleright
a_{(2)}), \, h_{(3)} \triangleleft a_{(3)}\Bigl) =
f(g_{(1)}, \, h_{(1)})[(g_{(2)} \cdot h_{(2)}) \triangleright a]$\\
\item[(BE5)] $\Bigl(g_{(1)} \triangleright f(h_{(1)}, \,
l_{(1)})\Bigl) f\Bigl(g_{(2)} \triangleleft f(h_{(2)}, \,
l_{(2)}), \, h_{(3)} \cdot l_{(3)}\Bigl) =
f(g_{(1)}, \, h_{(1)})f(g_{(2)} \cdot h_{(2)}, \, l)$\\
\item[(BE6)] $g_{(1)} \triangleleft a_{(1)} \otimes g_{(2)}
\triangleright a_{(2)} = g_{(2)} \triangleleft a_{(2)} \otimes
g_{(1)} \triangleright a_{(1)}$\\
\item[(BE7)] $g_{(1)} \cdot h_{(1)} \otimes f(g_{(2)}, \, h_{(2)})
= g_{(2)} \cdot h_{(2)} \otimes f(g_{(1)}, \, h_{(1)})$
\end{enumerate}
for all $g$, $h$, $l \in H$ and $a$, $b \in A$. In this case
$\Omega(A) = (H, \triangleleft, \, \triangleright, \, f)$ is
called a \emph{bialgebra extending structure} of $A$. A bialgebra
extending structure $\Omega(A) = (H, \triangleleft,
\triangleright, f)$ is called a \emph{Hopf algebra extending
structure} of A if $A \ltimes H$ has an antipode.

If one of the components $f$ or $\lhd$ of a bialgebra extending
structure $\bigl(H, \triangleleft, \triangleright, f \bigl)$ is
trivial, the unified product contains as special cases the
bicrossed product of bialgebras or the crossed product
\cite[Exmples 2.5]{am3}. Now, if $\rhd$ is the trivial action,
then a new product that will we used later on appears as a special
case of the unified product.

\begin{examples}\exlabel{3exemple}
$(1)$ Let $A$ be a bialgebra and $\Omega(A) = \bigl(H,
\triangleleft, \triangleright, f \bigl)$ an extending datum of $A$
such that the action $\rhd $ is trivial, that is $h \rhd a =
\varepsilon_H (h) a$, for all $h\in H$ and $a\in A$. Then
$\Omega(A) = \bigl(H, \triangleleft, \triangleright, f \bigl)$ is
a bialgebra extending structure of $A$ if and only if $\Delta_{H}
: H \to H\otimes H$ and $\varepsilon_{H} : H \to k$ are
$k$-algebra maps, $(H, \lhd)$ is a right $A$-module algebra and
the following compatibilities hold:
\begin{enumerate}
\item[(1)] $(g\cdot h)\cdot l = \bigl(g \triangleleft
f(h_{(1)}, \, l_{(1)})\bigl)\cdot (h_{(2)}\cdot l_{(2)})$\\
\item[(2)] $a_{(1)} f\bigl(g \lhd a_{(2)}, \, h \triangleleft
a_{(3)}\bigl) = f(g, \, h) a$\\
\item[(3)] $f(h_{(1)}, \, l_{(1)}) f\bigl(g \triangleleft
f(h_{(2)}, \, l_{(2)}), \, h_{(3)} \cdot l_{(3)}\bigl) =
f(g_{(1)}, \, h_{(1)})f(g_{(2)} \cdot h_{(2)}, \, l)$\\
\item[(4)] $g \triangleleft a_{(1)} \otimes a_{(2)} = g
\triangleleft a_{(2)} \otimes a_{(1)}$\\
\item[(5)] $g_{(1)} \cdot h_{(1)} \otimes f(g_{(2)}, \, h_{(2)}) =
g_{(2)} \cdot h_{(2)} \otimes f(g_{(1)}, \, h_{(1)})$
\end{enumerate}
for all $g$, $h$, $l \in H$ and $a\in A$. The unified product
$A\ltimes H$ associated to the bialgebra extending structure
$\Omega(A) = \bigl(H, \triangleleft, \triangleright, f \bigl)$
with $\rhd$ the trivial action will be denoted by $ A \lozenge H$
and will be called the \emph{twisted product} of $A$ and
$\Omega(A)$. In this bialgebra extending datum we shall omit the
action $\rhd$ and denote it by $\Omega(A) = \bigl(H, \lhd, f
\bigl)$. The bialgebra extending structure associated to
$\Omega(A) = \bigl(H, \lhd, f \bigl)$ will be called \emph{twisted
bialgebra extending structure} of $A$. Thus $A \lozenge H = A
\otimes H$ as a $k$-module with the multiplication given by:
\begin{equation}\eqlabel{101}
(a \lozenge h)\bullet(c \lozenge g) := a c_{(1)}
f\bigl(h_{(1)}\triangleleft c_{(2)}, \, g_{(1)}\bigl) \, \lozenge
\,  (h_{(2)}\triangleleft c_{(3)}) \cdot g_{(2)}
\end{equation}
for all $a$, $c \in A$, $h$ and $g\in H$, where we denoted $a
\otimes h \in A \otimes H$ by $a \lozenge h$. The twisted product
$A \lozenge H$ is a bialgebra with the tensor product of
coalgebras.

$(2)$ The right version of the smash product of Hopf algebras as
defined in \cite{molnar} is a special case of the twisted product
$A \lozenge H$, corresponding to the trivial cocycle $f$. We
explain this briefly in what follows. Let $\Omega(A) = \bigl(H,
\triangleleft, \triangleright, f \bigl)$ be an extending datum of
a bialgebra $A$ such that the action $\rhd $ and the cocycle $f$
are trivial. Then $\Omega(A) = \bigl(H, \triangleleft,
\triangleright, f \bigl)$ is a bialgebra extending structure of
$A$ if and only if $H$ is a bialgebra, $(H, \lhd)$ is a right
$A$-module bialgebra such that
$$
g \triangleleft a_{(1)} \otimes a_{(2)} = g \triangleleft a_{(2)}
\otimes a_{(1)}
$$
for all $g \in H$ and $a\in A$. In this case, the associated
unified product $A\ltimes H = A \# H$, is the right version of the
smash product of bialgebras introduced in \cite{molnar} in the
cocommutative case. Thus $A \# H = A \otimes H$ as a $k$-module
with the multiplication given by:
\begin{equation}\eqlabel{1011}
(a \# h)\bullet(c \# g) := a c_{(1)}  \, \# \,
(h_{(2)}\triangleleft c_{(2)}) g_{(2)}
\end{equation}
for all $a$, $c \in A$, $h$ and $g\in H$, where we denoted $a
\otimes h \in A \otimes H$ by $a \# h$.
\end{examples}

For a future use we shall recall the main characterization of a
unified product given in \cite[Theorem 2.7]{am3}.

\begin{theorem}\thlabel{2}
Let $E$ be a bialgebra, $A \subseteq E$ a subbialgebra, $H
\subseteq E$ a subcoalgebra such that $1_{E} \in H$ and the
multiplication map $u: A\otimes H \rightarrow E$, $u(a \otimes h)
= ah $, for all $a\in A$, $h\in H$ is bijective. Then, there
exists $\Omega(A)=(H, \triangleleft, \triangleright, f)$ a
bialgebra extending structure of $A$ such that $u: A \ltimes H
\rightarrow E$, $u(a \ltimes h) = ah$, for all $a\in A$ and $h\in
H$ is an isomorphism of bialgebras.

Explicitly, the actions, the cocycle and the multiplication of
$\Omega(A)$ are given by the formulas:
\begin{eqnarray}
\triangleright: H \otimes A \rightarrow A , \qquad \triangleright
&:=& (Id \otimes \varepsilon_{H}) \circ \mu \eqlabel{prima}\\
\triangleleft: H \otimes A \rightarrow H , \qquad \triangleleft
&:=& (\varepsilon_{A} \otimes Id) \circ \mu \eqlabel{adoua}\\
f: H\otimes H \rightarrow A , \qquad f &:=& (Id \otimes
\varepsilon_{H}) \circ \nu \eqlabel{22} \\
\cdot : H \otimes H \rightarrow H , \qquad \cdot &:=&
(\varepsilon_{A} \otimes Id) \circ \nu \eqlabel{23}
\end{eqnarray}
where
$$
\mu : H\otimes A \to A\ot H, \quad \mu (h\ot a) := u^{-1} (ha)
$$
$$
\nu : H\ot H \to A\ot H, \quad \nu (h\ot g) := u^{-1} (hg)
$$
for all $h$, $g \in H$ and $a\in A$.
\end{theorem}

\section{Unifying products versus split extensions of Hopf algebras}
For any bialgebra map $i: A\to E$, $E$ will be viewed as a left
$A$-module via $i$, that is $a\cdot x = i(a) x$, for all $a\in A$
and $x\in E$. We shall adopt a definition from \cite{AD} due in
the context of Hopf algebra maps.

\begin{definition}
Let $A$ and $E$ be two bialgebras. A coalgebra
map $\pi : E \to A$ is called \emph{normal} if the space
$$
\{ x \in E \, | \, \pi (x_{(1)}) \ot x_{(2)} = 1_A \ot x \}
$$
is a subcoalgebra of $E$.
\end{definition}

Let $G$ be a finite group, $H \leq G$ a subgroup of $G$ and
$k[G]^*$ the Hopf algebra of functions on $G$. Then the
restriction morphism $k[G]^* \to k[H]^*$ is a normal morphism if
and only if $H$ is a normal subgroup of $G$ (\cite{AD}).

The main properties of the bialgebra extension $ A \subset A
\ltimes H$ are given by the following:

\begin{proposition}\prlabel{3.1.1}
Let $A$ be a bialgebra, $\Omega(A) = \bigl(H, \triangleleft, \,
\triangleright, \, f \bigl)$ a bialgebra extending structure of
$A$ and the $k$-linear maps:
$$
i_A : A \to A  \ltimes H, \quad i_A (a) = a  \ltimes 1_H, \qquad
\pi_A : A  \ltimes H \to A, \quad \pi_A (a \ltimes h) =
\varepsilon_H (h) a
$$
for all $a \in A$, $h\in H$. Then:
\begin{enumerate}
\item $i_A$ is a bialgebra map, $\pi_A$ is a normal left
$A$-module coalgebra morphism and $\pi_A \circ i_A = {\rm Id}_A$.

\item $\pi_A$ is a right $A$-module map if and only if $\rhd$ is
the trivial action.

\item $\pi_A$ is a bialgebra map if and only if $\rhd$ and $f$ are
the trivial maps, i.e. the unified product $A \ltimes H = A \# H$,
the right version of the smash product of bialgebras.
\end{enumerate}
\end{proposition}

\begin{proof}
$(1)$ The fact that $i_A$ is a bialgebra map and $\pi_A$ is a
coalgebra map is straightforward. We show that $\pi_A$ is normal.
Indeed, the subspace
$$
\Bigl\{\sum_i a_i \ltimes h_i \in A  \ltimes H \,|\, \sum_i
a_{i_{(1)}} \varepsilon_H (h_{i_{(1)}}) \otimes a_{i_{(2)}}
\ltimes h_{i_{(2)}} = \sum_i 1_A \otimes a_i \ltimes h_i \Bigl\}
$$
is a subcoalgebra in $A \ltimes H$ since it can be identified with
$1_A \ltimes H$: more precisely, applying $\varepsilon_A$ on the
second position in the equality from the above subspace we obtain
that $ \sum_i a_i \ltimes h_i = 1_A \ltimes \sum_i \varepsilon_A
(a_i) h_i \in 1_A \ltimes H$, as needed. On the other hand
$$
\pi_A \bigl( a \cdot ( c \ltimes h) \bigl) = \pi_A \bigl( i_A (a)
\bullet ( c \ltimes h) \bigl) = \pi_A (a c \ltimes h) = ac \,
\varepsilon_H (h) = a \, \pi_A (c \ltimes h)
$$
for all $a$, $c\in A$ and $h\in H$, i.e. $\pi_A$ is a left
$A$-module map.

$(2)$ and $(3)$ are proven by a straightforward computation.
\end{proof}

The next Theorem gives the converse of \prref{3.1.1} $(1)$ and
the generalization of \cite[Theorem 4.1]{molnar}. Our proof is based on the
factorization \thref{2} and the fundamental theorem of
Hopf-modules.

\begin{theorem}\thlabel{3split}
Let $i: A\to E$ be a Hopf algebra morphism such that there exists
$\pi : E \to A$ a normal left $A$-module coalgebra morphism for
which $\pi \circ i = {\rm Id}_A$. Let
$$
H := \{ x \in E \, | \, \pi (x_{(1)}) \ot x_{(2)} = 1_A \ot x \}
$$
Then there exists a bialgebra extending structure $\Omega(A) =
\bigl(H, \triangleleft, \, \triangleright, \, f \bigl)$ of $A$,
where the multiplication on $H$, the actions $\rhd$, $\lhd$ and
the cocycle $f$ are given by the formulas:
\begin{eqnarray*}
h \cdot g &:=& i \Bigl ( S_A \bigl (  \pi ( h_{(1)} g_{(1)} )\bigl
) \Bigl)  h_{(2)} g_{(2)}, \quad \,\,\,\,\,\,\,\,\,\, f (h, g) := \pi (hg) \\
h \lhd a &:=& i \Bigl ( S_A \bigl (  \pi ( h_{(1)} i(a_{(1)})
)\bigl ) \Bigl)  h_{(2)} i(a_{(2)}), \quad h \rhd a := \pi \bigl(h
i(a) \bigl)
\end{eqnarray*}
for all $h$, $g\in H$, $a\in A$ such that
$$
\varphi : A \ltimes H \to E, \quad \varphi ( a \ltimes h) = i(a) h
$$
for all $a\in A$ and $h\in H$ is an isomorphism of Hopf algebras.
\end{theorem}

\begin{proof}
There are two possible ways to prove the above theorem: the first
one is to give a standard proof in the way this type of theorems
are usually proved in Hopf algebra theory, by showing that all
compatibility conditions $(BE1) - (BE7)$ hold and then to prove
that $\varphi$ is an isomorphism of Hopf algebras. We prefer
however a more direct approach which relies on \thref{2}: it has
the advantage of making more transparent to the reader the way we
obtained the formulas which define the bialgebra extending
structure $\Omega(A)$. First we observe that $1_E \in H$ since
$\pi (1_E) = \pi (i (1_A)) = 1_A$ and $H$ is a subcoalgebra of $E$
as $\pi$ is normal. $E$ has a structure of left-left $A$-Hopf
module via the left $A$-action and the left $A$-coaction given by
$$
a \cdot x := i(a)x , \quad \rho (x) = x_{<-1>} \ot x_{<0>} := \pi
(x_{(1)}) \ot x_{(2)}
$$
for all $a\in A$ and $x\in E$. Indeed let us prove the Hopf module
compatibility condition:
\begin{eqnarray*}
\rho (a\cdot x) &=& \rho \bigl( i(a) x \bigl) = \pi \bigl(
i(a_{(1)}) x_{(1)} \bigl) \ot i(a_{(2)}) x_{(2)}\\
&=& \pi \bigl( a_{(1)} \cdot x_{(1)} \bigl) \ot i(a_{(2)}) x_{(2)}
= a_{(1)} \pi ( x_{(1)}) \ot a_{(2)} \cdot x_{(2)} \\
&=& a_{(1)} x_{<-1>} \ot a_{(2)} \cdot x_{<0>}
\end{eqnarray*}
for all $a\in A$ and $x\in E$, i.e. $E$ is a left-left $A$-Hopf
module. We also note that $ H = E ^{{\rm co} (A)}$. It follows
from the fundamental theorem of Hopf modules that the map
$$
\varphi : A \otimes H \to E, \quad \varphi ( a \otimes x) :=
a\cdot x = i(a) x
$$
for all $a\in A$ and $h\in H$ is an isomorphism of vector spaces
with the inverse given by
\begin{equation}\eqlabel{3.2.22}
\varphi^{-1} (x) := x_{<-2>} \ot S_A (x_{<-1>}) \cdot x_{<0>} =
\pi (x_{(1)}) \ot i \Bigl ( S_A \bigl (  \pi ( x_{(2)} )\bigl )
\Bigl ) x_{(3)}
\end{equation}
for all $x\in E$. Thus $E$ is a Hopf algebra that factorizes
through $i(A)\cong A$ and the subcoalgebra $H$. Now, from
\thref{2} we obtain that there exists a bialgebra extending
structure $\Omega(A) = \bigl(H, \triangleleft, \triangleright, f
\bigl)$ of $A$ such that $ \varphi : A \ltimes H \to E$, $\varphi
( a \ltimes h) = i(a) h$, for all $a\in A$ and $h\in H$ is an
isomorphism of Hopf algebras. Using \equref{3.2.22} the actions
$\rhd$, $\lhd$ given by the formulas \equref{prima},
\equref{adoua} take the explicit form:
\begin{eqnarray*}
h \rhd a &=& (Id_A \ot \varepsilon_H) \varphi^{-1} (h i(a) ) \\
&=& \pi \bigl(h_{(1)} i(a_{(1)}) \bigl) \varepsilon_H (h_{(2)})
\varepsilon_A (a_{(2)}) \varepsilon_H (h_{(3)}) \varepsilon_A
(a_{(3)}) = \pi \bigl(h i(a) \bigl)
\end{eqnarray*}
and
\begin{eqnarray*}
h \lhd a &=& (\varepsilon_A \ot Id_E ) \varphi^{-1} (h i(a) ) \\
&=& i \Bigl ( S_A \bigl (  \pi ( h_{(1)} i(a_{(1)}) )\bigl )
\Bigl)  h_{(2)} i(a_{(2)})
\end{eqnarray*}
for all $h\in H$ and $a\in A$. Finally, using once again
\equref{3.2.22}, the multiplication $\cdot$ on $H$ and the cocycle
$f$ given by \equref{22}, \equref{23} take the form
$$
h \cdot g := i \Bigl ( S_A \bigl (  \pi ( h_{(1)} g_{(1)} )\bigl )
\Bigl)  h_{(2)} g_{(2)}, \quad f (h, g) := \pi (hg)
$$
for all $h$, $g\in H$. The proof is now finished by \thref{2}.
\end{proof}

The next corollary covers the case in which the splitting map
$\pi$ is also a right $A$-module map. The version of
\coref{bimodsplit} below in which the normality condition of the
splitting map $\pi$ is dropped was proved in \cite[Theorem
3.64]{AMSt}. In this case, the input data of the construction of
the product that is used is called a dual Yetter-Drinfel'd
quadruple \cite[Definiton 3.59]{AMSt}, and consists of a system of
objects and maps satisfying eleven compatibility conditions.

\begin{corollary}\colabel{bimodsplit}
Let $i: A\to E$ be a Hopf algebra morphism such that there exists
$\pi : E \to A$ a normal $A$-bimodule coalgebra morphism such that
$\pi \circ i = {\rm Id}_A$. Then there exists $\Omega (A) = (H,
\lhd, f)$ a twisted bialgebra extending structure of $A$ such that
$$
\varphi : A \lozenge H \to E, \quad \varphi (a \lozenge h) = i(a)h
$$
for all $a\in A$ and $h\in H$ is an isomorphism of Hopf algebras.
\end{corollary}

\begin{proof} It follows from \thref{3split} that there
exists a bialgebra extending structure $\Omega(A) = \bigl(H,
\triangleleft, \, \triangleright, \, f \bigl)$ of $A$ such that
$$
\varphi : A \ltimes H \to E, \quad \varphi ( a \ltimes h) = i(a) h
$$
for all $a\in A$ and $h\in H$ is an isomorphism of Hopf algebras.
We observe that for any $h \in H = \{ x \in E \, | \, \pi
(x_{(1)}) \ot x_{(2)} = 1_A \ot x \}$, we have that $\pi (h) =
\varepsilon_H(h) 1_A$. Thus, as $\pi$ is also a right $A$-module
map, the action $\rhd$ is given by:
$$
h \rhd a = \pi \bigl(h i(a) \bigl) = \pi (h \cdot a) = \pi (h) a =
\varepsilon_H(h) a
$$
for all $h\in H$ and $a\in A$. So the action $\rhd$ is trivial and
hence the unified product $A \ltimes H$ is a twisted product $A
\lozenge H$.
\end{proof}

The following is a simplified and more transparent version of
\cite[Theorem 4.1]{molnar}. We use a minimal context: only the
concept of normality of a morphism in the sense of \cite{AD} is
used,  contrary to \cite[Definition 3.5]{molnar} where it was
taken as input data for \cite[Theorem 4.1]{molnar}.

\begin{corollary}\colabel{molnar}
Let $\pi : E \to A$ be a normal split epimorphism of Hopf
algebras. Then there exists an isomorphism of Hopf algebras $ E
\cong A \# H$, for some right $A$-module bialgebra $H$, where $A
\# H$ is the right version of the smash product of bialgebras.
\end{corollary}

\begin{proof}
Let $i : A \to E$ be a Hopf algebra map such that $\pi \circ i =
{\rm Id}_A$ and
$$
H = \{ x \in E \, | \, \pi (x_{(1)}) \ot x_{(2)} = 1_A \ot x \}
$$
which is a subcoalgebra in $E$ as $\pi$ is normal. In fact, as
$\pi$ is also an algebra map, $H$ is a Hopf subalgebra of $E$. We
note that $\pi (h) = \varepsilon_H(h) 1_A$, for all $h\in H$.
$\pi$ is also a $A$-bimodule map as $\pi (a \cdot h) = \pi (i (a)
h) = \pi (i (a)) \pi (h) = a \pi (h)$ and $\pi (h \cdot a) = \pi
(h i(a)) = \pi (h) a$, for all $a\in A$ and $h\in H$. Now we can
apply \thref{3split}. Thus, there exists a bialgebra extending
structure $\Omega(A) = \bigl(H, \triangleleft, \, \triangleright,
\, f \bigl)$ of $A$ such that
$$
\varphi : A \ltimes H \to E, \quad \varphi ( a \ltimes h) = i(a) h
$$
is an isomorphism of bialgebras. Moreover the action $\rhd$ is the
trivial one as $\pi$ is a right $A$-module map
(\coref{bimodsplit}). On the other hand, the cocycle $f$ as it was
defined in \thref{3split} is also the trivial one: $f (h, g) = \pi
(h g) = \pi (h) \pi (g) = \varepsilon_H(h) \varepsilon_H(g) 1_A$,
for all $h$, $g\in H$. Using once again the fact that $\pi$ is an
algebra map, the multiplication $\cdot$ on $H$ as it was defined
in \thref{3split} is exactly the one of $E$ (and this fits with
the fact that $H$ is a Hopf subalgebra of $E$). Finally, the right
action as it was defined in \thref{3split} takes the simplified
form: $h \lhd a = i \bigl( S_A (a_{(1)}) \bigl) \, h \, a_{(1)}$,
for all $h\in H$ and $a\in A$. Thus, the unified product $A
\ltimes H$ is the right version of the smash product of Hopf
algebras from \exref{3exemple} and $\varphi : A \# H \to E$ is an
isomorphism of Hopf algebras.
\end{proof}

From now on, $A$ will be a Hopf algebra and $\Omega(A) = \bigl(H,
\triangleleft, \, \triangleright, \, f \bigl)$ a bialgebra
extending structure of $A$. Let $i_A : A \to A  \ltimes H$, $i_A
(a) = a \ltimes 1_H$, for all $a\in A$ be the canonical bialgebra
morphism.

\begin{proposition}\prlabel{splitmono1}
Let $A$ be a Hopf algebra and $\Omega(A) = \bigl(H, \triangleleft,
\, \triangleright, \, f \bigl)$ a bialgebra extending structure of
$A$. Then $i_A : A \to A  \ltimes H$ is a split monomorphism in
the category of bialgebras if and only if there exists $\gamma: H
\rightarrow A$ a unitary coalgebra map such that
\begin{eqnarray}
h\triangleright a &=& \gamma(h_{(1)}) \, a_{(1)} \,
\gamma^{-1}(h_{(2)}\lhd a_{(2)} ) \eqlabel{iner1}\\
f(h,\,g) &=& \gamma(h_{(1)}) \, \gamma(g_{(1)}) \,
\gamma^{-1}(h_{(2)} \cdot g_{(2)}) \eqlabel{iner2}
\end{eqnarray}
for all $h$, $g\in H$ and $a\in A$, where $\gamma^{-1} = S_A \circ
\gamma$.
\end{proposition}

\begin{proof}
$i_{A}: A \to A  \ltimes H$ is a split monomorphism of bialgebras
if and only if there exists a bialgebra map $p: A \ltimes H
\rightarrow A$ such that $p (a \ltimes 1) = a$, for all $a \in A$.
Such a Hopf algebra map $p$ is given by:
$$
p(a \ltimes h) = p\bigl((a \ltimes 1_H )\bullet (1_A \ltimes h)\bigl) =
a \, p(1_A \# h)
$$
for all $a\in A$ and $h\in H$. We denote by $\gamma: H \rightarrow
A$, $\gamma(h):= p(1 \ltimes h)$, for all $h\in H$. Hence, such a
splitting map $p$ should be given by
$$
p = p_{\gamma} : A \ltimes H \rightarrow A, \quad p_{\gamma} (a
\ltimes h) = a \gamma (h)
$$
for all $a\in A$ and $h\in H$. First we note that $p_{\gamma}$ is
a coalgebra map if and only if $\gamma$ is a coalgebra map. Now,
we prove that $p_{\gamma}$ is an algebra map if and only if
$\gamma(1_H) = 1_A$ and the following two compatibilities are
fulfilled:
\begin{eqnarray}
\eqlabel{s1} \gamma(h) a &{=}& \bigl(h_{(1)} \triangleright
a_{(1)} \bigl) \,
\gamma(h_{(2)} \lhd a_{(2)} )\\
\eqlabel{s2} \gamma(h)\gamma(g) &{=}& f(h_{(1)}, \, g_{(1)}) \,
\gamma(h_{(2)} \cdot g_{(2)})
\end{eqnarray}
for all $h$, $g\in H$ and $a\in A$. Of course, $p_{\gamma} (1_A
\ltimes 1_H) = 1_A$ if and only $\gamma (1_H) = 1_A$. We assume
now that $\gamma (1_H) = 1_A$. Then, $p_{\gamma}$ is an algebra
map if and only if $p_{\gamma}(xy) = p_{\gamma}(x) p_{\gamma}(y)$,
for all $x$, $y \in T := \{a \ltimes 1_{H} ~|~ a \in A\} \cup
\{1_{A} \ltimes g ~|~ g \in H\}$, the set of generators as an
algebra of $A\ltimes H$ (see \cite{am3}). Now, for any $a$, $c\in
A$ and $h\in H$ we have:
$$
p_{\gamma} \bigl( ( a \ltimes 1_H) \bullet (c \ltimes 1_H) \bigl) = a c =
p_{\gamma} (a \ltimes 1_H) p_{\gamma} (c \ltimes 1_H)
$$
and
$$
p_{\gamma} \bigl( ( a \ltimes 1_H) \bullet (1_A \ltimes h) \bigl) = a
\gamma (h) = p_{\gamma} (a \ltimes 1_H) p_{\gamma} (1_A \ltimes h)
$$
On the other hand, it is straightforward to see that
$$
p_{\gamma} \bigl( ( 1_A \ltimes h) \bullet (a \ltimes 1_H) \bigl) =
p_{\gamma} ( 1_A \ltimes h) p_{\gamma} ( a \ltimes 1_H)
$$
if and only if \equref{s1} holds and
$$
p_{\gamma} \bigl( ( 1_A \ltimes h) \bullet (1_A \ltimes g) \bigl) =
p_{\gamma} ( 1_A \ltimes h) p_{\gamma} ( 1_A \ltimes g)
$$
if and only if \equref{s2} holds. Thus, we have proved that
$p_{\gamma}$ is a  bialgebra map if and only if $\gamma : H \to A$
is a unitary coalgebra map and the compatibility conditions
\equref{s1} and \equref{s2} are fulfilled. Being a coalgebra map,
$\gamma$ is invertible in convolution with the inverse given by
$\gamma^{-1} = S_A \circ \gamma$. We observe that \equref{s1} is
equivalent to \equref{iner1} while \equref{s2} is equivalent to
\equref{iner2} and the proof is finished.
\end{proof}

Suppose now that we are in the setting of \prref{splitmono1}. The
multiplication on the unified product $A \ltimes H$ given by
\equref{10} takes the following form
\begin{equation}\eqlabel{10a}
(a \ltimes h)\bullet(c \ltimes g) = a \gamma (h_{(1)}) c_{(1)}
\gamma (g_{(1)}) \gamma^{-1} \Bigl ( (h_{(2)} \lhd c_{(2)} ) \cdot
g_{(2)} \Bigl) \ltimes (h_{(3)} \lhd c_{(3)} ) \cdot g_{(3)}
\end{equation}
for all $a$, $c\in A$ and $h$, $g\in H$, which is still very
difficult to deal with. Next we shall give another equivalent
description of the bialgebra structure on this special unified
product in which the multiplication has a less complicated form.
Let $A \circledast H = A \ot H$, as a $k$-module with the unit
$1_A \circledast 1_H$ and the following structures:
\begin{eqnarray}
(a \circledast h)\cdot (c \circledast g) &:=& a c_{(1)}
\circledast \Bigl ( h \lhd \bigl( c_{(2)} \gamma^{-1} (g_{(1)}
\bigl) \Bigl) \cdot
g_{(2)} \eqlabel{10b} \\
\Delta_{A \circledast H} \, (a  \circledast h) &:=& a_{(1)}
\circledast h_{(2)} \ot a_{(2)} \gamma^{-1} (h_{(1)}) \gamma
(h_{(3)}) \circledast h_{(4)} \eqlabel{10c}\\
\varepsilon_{A \circledast H} \, (a  \circledast h) &:=&
\varepsilon_A (a) \varepsilon_H (h) \eqlabel{10d}
\end{eqnarray}
for all $a$, $c\in A$ and $h$, $g\in H$, where we denoted $a
\otimes h \in A \otimes H$ by $a \circledast h$.

The object $A \circledast H$ introduced above is an
interesting deformation of the smash product of bialgebras $A\# H$ of
\exref{3exemple} that can be recovered in the case that $\gamma : H \to A$
is the trivial map, that is $\gamma (h) = \varepsilon_H (h) 1_A$,
for all $h\in H$. If $H$ is cocommutative then the comultiplication
given by \equref{10c} is just the tensor product of coalgebras.

\begin{proposition}\prlabel{produsulciudat}
Let $\Omega(A) = \bigl(H, \triangleleft, \, \triangleright, \, f
\bigl)$ be a bialgebra extending structure of a Hopf algebra $A$
and $\gamma : H \to A$ be a unitary coalgebra map such that
\equref{iner1} and \equref{iner2} hold. Then
$$
\varphi : A \ltimes H \to A \circledast H, \quad \varphi ( a
\ltimes h) := a \gamma (h_{(1)}) \circledast h_{(2)}
$$
for all $a\in A$ and $h\in H$ is an isomorphism of bialgebras.
\end{proposition}

\begin{proof}
The map $\varphi : A \ltimes H \to A \ot H $ is bijective with the
inverse $ \varphi^{-1} : A \ot H \to A \ltimes H$ given by
$\varphi^{-1} (a \ot h) = a \gamma^{-1} (h_{(1)}) \ltimes
h_{(2)}$, for all $a \in A$ and $h \in H$, where $\gamma^{-1} =
S_A \circ \gamma$. Below we shall use the fact that $\gamma^{-1} :
H \to A$ is an antimorphism of coalgebras, i.e. $\Delta \bigl(
\gamma^{-1} (h) \bigl) = \gamma^{-1} (h_{(2)}) \ot \gamma^{-1}
(h_{(1)})$, for all $h\in H$. The unique algebra structure that
can be defined on the $k$-module $A \otimes H$ such that $\varphi
: A \ltimes H \to A \ot H$ becomes an isomorphism of algebras is
given by:
\begin{eqnarray*}
(a \ot h) \cdot (c \ot g) &=& \varphi \Bigl ( \varphi^{-1} (a\ot
h) \bullet \varphi^{-1} (c\ot h) \Bigl) \\
&=& \varphi \Bigl ( \bigl( a\gamma^{-1} (h_{(1)}) \ltimes h_{(2)}
\bigl ) \bullet \bigl( c\gamma^{-1} (g_{(1)}) \ltimes g_{(2)}
\bigl) \Bigl)\\
&\stackrel{\equref{10a}} =& \varphi \Bigl ( a \gamma^{-1}(h_{(1)})
\gamma (h_{(2)}) c_{(1)} \gamma^{-1}(g_{(3)}) \gamma (g_{(4)})\\
&& \gamma^{-1}  [ ( h_{(3)} \lhd (c_{(2)} \gamma^{-1}(g_{(2)}) ) )
\cdot g_{(5)}] \ltimes [ ( h_{(4)} \lhd (c_{(3)}
\gamma^{-1}(g_{(1)}) ) ) \cdot g_{(6)}] \Bigl)\\
&=& \varphi \Bigl ( a c_{(1)} \gamma^{-1}  [ ( h_{(1)} \lhd
(c_{(2)} \gamma^{-1}(g_{(2)}) ) ) \cdot g_{(3)}] \\
&&\ltimes ( h_{(2)} \lhd (c_{(3)} \gamma^{-1}(g_{(1)}) ) ) \cdot
g_{(4)} \Bigl)
\end{eqnarray*}
\begin{eqnarray*}
&=& a c_{(1)} \gamma^{-1}  [ ( h_{(1)} \lhd (c_{(2)}
\gamma^{-1}(g_{(3)}) ) ) \cdot g_{(4)}] \\
&& \gamma [ ( h_{(2)} \lhd (c_{(3)} \gamma^{-1}(g_{(2)}) ) ) \cdot
g_{(5)}] \ot ( h_{(3)} \lhd (c_{(4)} \gamma^{-1}(g_{(1)}) ) )
\cdot g_{(6)}\\
&=& a c_{(1)} \ot \Bigl ( h \lhd \bigl( c_{(2)} \gamma^{-1}
(g_{(1)} \bigl) \Bigl) \cdot g_{(2)}
\end{eqnarray*}
which is exactly the multiplication defined by \equref{10b} on $A
\circledast H = A \ot H$. Thus, we have proved that $\varphi : A
\ltimes H \to A \circledast H$ is an isomorphism of algebras.

It remains to prove that $\varphi : A \ltimes H \to A \circledast
H$ is also a coalgebra map. For any $a \in A$ and $h\in H$ we
have:
\begin{eqnarray*}
\Delta_{A \circledast H} \bigl( \varphi (a \ltimes h ) \bigl) &=&
\Delta_{A \circledast H} \bigl( a \gamma (h_{(1)}) \circledast
h_{(2)} \bigl) \\
&\stackrel{\equref{10c}} =&  a_{(1)} \gamma (h_{(1)}) \circledast
h_{(4)} \ot a_{(2)} \gamma (h_{(2)})
\gamma^{-1} (h_{(3)}) \gamma (h_{(5)}) \circledast h_{(6)}\\
&=& a_{(1)} \gamma (h_{(1)}) \circledast h_{(2)} \ot a_{(2)}
\gamma (h_{(3)}) \circledast h_{(4)}\\
&=& \varphi ( a_{(1)} \ltimes h_{(1)} ) \ot \varphi ( a_{(2)}
\ltimes h_{(2)} )
\end{eqnarray*}
i.e. $\varphi: A \ltimes H \to A \circledast H$ is a coalgebra
map, as needed. By assumption,  $A \ltimes H$ is a bialgebra, thus
$A \circledast H$ is a bialgebra and $\varphi$ is an isomorphism
of bialgebras.
\end{proof}

\prref{splitmono1} provides a method to construct bialgebra
extending structures, and thus unified products, starting only
with a right action $\lhd$ and a unitary coalgebra map $\gamma : H
\to A$.

\begin{theorem}\thlabel{construnifbirp}
Let $A$ be a Hopf algebra, $H$ a unitary not necessarily
associative bialgebra such that $(H, \lhd)$ is a right $A$-module
coalgebra with $1_A \triangleleft a = \varepsilon_A (a) 1_H$, for
all $a \in A$. Let $\gamma : H \to A$ be a unitary coalgebra map
and define
\begin{eqnarray}
\rhd_{\gamma} : H\ot A \to A, \quad h \rhd_{\gamma} a : &=&
\gamma(h_{(1)}) \, a_{(1)} \, \gamma^{-1}(h_{(2)}\lhd a_{(2)} ) \eqlabel{iner11}\\
f_{\gamma} : H\ot H \to A, \quad f_{\gamma} (h,\,g) : &=&
\gamma(h_{(1)}) \, \gamma(g_{(1)}) \, \gamma^{-1}(h_{(2)} \cdot
g_{(2)}) \eqlabel{iner22}
\end{eqnarray}
for all $h$, $g\in H$ and $a\in A$. Assume that the compatibility
conditions $(BE1)$, $(BE3)$, $(BE6)$ and $(BE7)$ hold for
$\rhd_{\gamma}$ and $f_{\gamma}$.

Then $\Omega(A) = (H, \triangleleft, \, \triangleright_{\gamma},
\, f_{\gamma} )$ is a bialgebra extending structure of $A$ and
there exists an isomorphism of bialgebras $A \ltimes H \cong L
\ast A$, where $L \ast A$ is the Radford biproduct for a bialgebra
$L$ in the category ${}_{A}^{A}{}{\mathcal YD}$ of
Yetter-Drinfel'd modules.
\end{theorem}

\begin{proof}
First we have to prove that $\rhd_{\gamma}$ and $f_{\gamma}$ are
coalgebra maps. Using the fact that $\gamma^{-1} = S_A \circ
\gamma$ is an antimorphism of coalgebras we have:
\begin{eqnarray*}
\Delta_A (h \rhd_{\gamma} a) &=& \gamma (h_{(1)}) a_{(1)}
\gamma^{-1} ( h_{(4)} \lhd a_{(4)} ) \ot \gamma (h_{(2)}) a_{(2)}
\gamma^{-1} ( h_{(3)} \lhd a_{(3)} )\\
&=& \gamma (h_{(1)}) a_{(1)} \gamma^{-1} ( h_{(3)} \lhd a_{(3)} )
\ot h_{(2)} \rhd_{\gamma} a_{(2)}\\
&\stackrel{(BE6)} =& \gamma (h_{(1)}) a_{(1)} \gamma^{-1} (
h_{(2)} \lhd a_{(2)} ) \ot h_{(3)} \rhd_{\gamma} a_{(3)}\\
&=& h_{(1)} \rhd_{\gamma} a_{(1)} \ot h_{(2)} \rhd_{\gamma}
a_{(2)}
\end{eqnarray*}
for all $h\in H$ and $a\in A$, i.e. $\rhd_{\gamma}$ is a coalgebra
map. On the other hand, we have:
\begin{eqnarray*}
\Delta_A (f_{\gamma} (h, g) ) &=& \gamma(h_{(1)}) \,
\gamma(g_{(1)}) \, \gamma^{-1}(h_{(4)} \cdot g_{(4)}) \ot
\gamma(h_{(2)}) \, \gamma(g_{(2)}) \, \gamma^{-1}(h_{(3)} \cdot
g_{(3)})\\
&=& \gamma(h_{(1)}) \, \gamma(g_{(1)}) \, \gamma^{-1}(h_{(3)}
\cdot g_{(3)}) \ot f_{\gamma} (h_{(2)}, \, g_{(2)} )\\
&\stackrel{(BE7)} =& f_{\gamma} (h_{(1)}, \, g_{(1)} ) \ot
f_{\gamma} (h_{(2)}, \, g_{(2)} )\\
&=& (f_{\gamma} \ot f_{\gamma} ) \Delta_{H\ot H} (h\ot g)
\end{eqnarray*}
for all $h$, $g\in H$, that is $f_{\gamma}$ is a coalgebra map.

It remains to prove that the compatibility conditions $(BE2)$,
$(BE4)$ and $(BE5)$ hold for $\rhd_{\gamma}$ and $f_{\gamma}$. For
$a$, $b\in A$ and $g\in H$ we have:
\begin{eqnarray*}
g \rhd_{\gamma} (a b) &=& \gamma( g_{(1)} ) a_{(1)} b_{(1)}
\gamma^{-1} \bigl( g_{(2)} \lhd (a_{(2)}b_{(2)}) \bigl)\\
&=& \gamma (g_{(1)}) a_{(1)} \gamma^{-1} (g_{(2)} \lhd a_{(2)})
\gamma (g_{(3)} \lhd a_{(3)} ) b_{(1)} \gamma^{-1} \bigl( g_{(4)}
\lhd (a_{(4)}b_{(2)}) \bigl)\\
&=& (g_{(1)} \triangleright_{\gamma}
a_{(1)})[(g_{(2)}\triangleleft a_{(2)})\triangleright_{\gamma} b]
\end{eqnarray*}
i.e. $(BE2)$ holds. We denote by LHS (resp. RHS) the left (resp.
right) hand side of $(BE4)$. We have:
\begin{eqnarray*}
LHS &=& \gamma(g_{(1)}) (h_{(1)} \rhd_{\gamma} a_{(1)} )
\gamma^{-1} \Bigl ( g_{(2)}\lhd ( h_{(2)} \rhd_{\gamma} a_{(2)})
\Bigl) \gamma \Bigl ( g_{(3)}\lhd ( h_{(3)} \rhd_{\gamma} a_{(3)})
\Bigl)\\
&& \gamma (h_{(5)} \lhd a_{(5)}) \gamma^{-1} \Bigl ( \bigl
(g_{(4)}\lhd ( h_{(4)} \rhd_{\gamma} a_{(4)}) \bigl) \cdot
(h_{(6)} \lhd a_{(6)}) \Bigl)\\
&=& \gamma(g_{(1)}) (h_{(1)} \rhd_{\gamma} a_{(1)} ) \gamma
(h_{(3)} \lhd a_{(3)})  \gamma^{-1} \Bigl ( \bigl (g_{(2)}\lhd (
h_{(2)} \rhd_{\gamma} a_{(2)}) \bigl) \cdot
(h_{(4)} \lhd a_{(4)}) \Bigl)\\
&\stackrel{(BE6)} =& \gamma(g_{(1)}) (h_{(1)} \rhd_{\gamma}
a_{(1)} ) \gamma (h_{(2)} \lhd a_{(2)})  \gamma^{-1} \Bigl ( \bigl
(g_{(2)}\lhd ( h_{(3)} \rhd_{\gamma} a_{(3)}) \bigl) \cdot
(h_{(4)} \lhd a_{(4)}) \Bigl)\\
&\stackrel{(BE3)} =& \gamma(g_{(1)}) (h_{(1)} \rhd_{\gamma}
a_{(1)} ) \gamma (h_{(2)} \lhd a_{(2)})  \gamma^{-1} \Bigl ( \bigl
(g_{(2)}\cdot h_{(3)} \bigl) \lhd \, a_{(3)} \Bigl )\\
&\stackrel{\equref{iner11}} =& \gamma(g_{(1)}) \gamma(h_{(1)})
a_{(1)} \gamma^{-1} \Bigl ( \bigl (g_{(2)}\cdot h_{(2)} \bigl)
\lhd \, a_{(2)} \Bigl )\\
&=& RHS
\end{eqnarray*}
for all $a\in A$, $h$, $g\in H$, i.e. $(BE4)$ holds. Now, for $g$,
$h$ and $l\in H$ the right hand side of $(BE5)$ takes the
following form:
$$
RHS = \gamma(g_{(1)})\gamma(h_{(1)})\gamma(l_{(1)}) \gamma^{-1}
\bigl( (g_{(2)} \cdot h_{(2)})\cdot l_{(2)} \bigl)
$$
while the left hand side of $(BE5)$ is
\begin{eqnarray*}
LHS &=& \gamma(g_{(1)}) f_{\gamma} (h_{(1)}, l_{(1)})
\gamma(h_{(3)} \cdot l_{(3)}) \gamma^{-1} \Bigl ( \bigl( g_{(2)}
\lhd f_{\gamma} (h_{(2)}, l_{(2)}) \bigl) \cdot (h_{(4)}\cdot
l_{(4)}) \Bigl )\\
&\stackrel{(BE7)} =& \gamma(g_{(1)}) f_{\gamma} (h_{(1)}, l_{(1)})
\gamma(h_{(2)} \cdot l_{(2)}) \gamma^{-1} \Bigl ( \bigl( g_{(2)}
\lhd f_{\gamma} (h_{(3)}, l_{(3)}) \bigl) \cdot (h_{(4)}\cdot
l_{(4)}) \Bigl )\\
&\stackrel{(BE1)} =& \gamma(g_{(1)}) f_{\gamma} (h_{(1)}, l_{(1)})
\gamma(h_{(2)} \cdot l_{(2)}) \gamma^{-1} \bigl( (g_{(2)} \cdot
h_{(3)})\cdot l_{(3)} \bigl)\\
&\stackrel{\equref{iner22}} =&
\gamma(g_{(1)})\gamma(h_{(1)})\gamma(l_{(1)}) \gamma^{-1} \bigl(
(g_{(2)} \cdot h_{(2)})\cdot l_{(2)} \bigl) = RHS
\end{eqnarray*}
hence $(BE5)$ also holds and thus $\Omega(A) = (H, \triangleleft,
\, \triangleright_{\gamma}, \, f_{\gamma} )$ is a bialgebra
extending structure of $A$.

For the final part we use \prref{splitmono1} as $p : A \ltimes H
\rightarrow A$, $p (a \ltimes h) = a \gamma (h)$ is a bialgebra
map that splits $i_A: A \to A\ltimes H$. Thus, it follows from
\cite[Theorem 3]{radford} that there exists an isomorphism of
bialgebras $A \ltimes H \cong L \ast A$, where $(L,
\rightharpoonup ,\rho_L)$ is a bialgebra in ${}_{A}^{A}{}{\mathcal
YD}$ as follows:
$$
L = \{ \sum_i a_i \ltimes h_i \in A  \ltimes H \, | \,
\sum_i a_{i_{(1)}} \ltimes h_{i_{(1)}} \ot a_{i_{(2)}} \gamma (h_{i_{(2)}}) = \sum_i a_i \ltimes h_i \ot 1_A \}
$$
with the coalgebra structure described below:
\begin{eqnarray*}
\Delta_L ( \sum_i a_i \ltimes h_i ) &=& \sum_i (a_{i_{(1)}} \ltimes h_{i_{(1)}})
\bullet \bigl( 1_A \ltimes S_A ( a_{i_{(2)}} \gamma (h_{i_{(2)}}))  \bigl) \, \ot \,  a_{i_{(3)}} \ltimes h_{i_{(3)}}\\
\varepsilon_L ( \sum_i a_i \ltimes h_i  ) &=& \sum_i \varepsilon_A (a_i) \varepsilon_H (h_i)
\end{eqnarray*}
for all $\sum_i a_i \ltimes h_i \in L$ and the structure of an
object in ${}_{A}^{A}{}{\mathcal YD}$ given by
\begin{eqnarray*}
a \rightharpoonup \sum_i a_i \ltimes h_i &=&
\sum_i ( a_{(1)} a_i \ltimes h_i) \bullet (1_A \ltimes S_A (a_{(2)}))\\
\rho_L (\sum_i a_i \ltimes h_i) &=& \sum_i a_{i_{(1)}} \gamma (h_{i_{(1)}}) \ot a_{i_{(2)}} \ltimes h_{i_{(2)}}
\end{eqnarray*}
for all $a \in A$ and $\sum_i a_i \ltimes h_i \in L$.
\end{proof}

Using \thref{construnifbirp} we shall construct an example of an
unified product, that is also isomorphic to a biproduct, starting
with a minimal set of data.

\begin{example}\exlabel{ultex}
Let $G$ be a group, $(X, 1_X)$ be a pointed set, such that there
exists a binary operation $ \cdot : X \times X \to X$ having $1_X$
as a unit.

Let $ \lhd : X \times G \to X$ be a map such that $(X, \lhd)$ is a
right $G$-set with $1_X \triangleleft g = 1_X$, for all $g \in G$
and $\gamma : X \to G$ a map with $\gamma (1_X) = 1_G$ such that
the following compatibilities hold
\begin{eqnarray}
(x \cdot y) \cdot z &=& \Bigl ( x \lhd \bigl( \gamma(y) \gamma (z)
\gamma (y\cdot z) ^{-1} \bigl) \Bigl ) \cdot (y \cdot z)
\eqlabel{be1}\\
(x \cdot y) \lhd  g &=& \Bigl( x \lhd \bigl( \gamma (y) \, g \,
\gamma ( y \lhd g )^{-1} \bigl) \Bigl) \cdot (y \lhd g)
\eqlabel{be3}
\end{eqnarray}
for all $g\in G$, $x$, $y$, $z\in X$.

Let $A := k[G]$. Then $(H := k[X], \, \lhd, \, \rhd_{\gamma}, \,
f_{\gamma} )$ is a bialgebra extending structure of $k[G]$, where
$\rhd_{\gamma}$, $f_{\gamma}$ are given by \equref{iner11},
\equref{iner22}, that is
$$
x \rhd_{\gamma} g = \gamma (x) g \gamma (x \lhd g) ^{-1}, \quad
f_{\gamma} (x, y) = \gamma (x) \gamma (y) \gamma (x \cdot y)^{-1}
$$
for all $x$, $y\in X$ and $g\in G$.

Indeed, \equref{be1} and \equref{be3} show that the compatibility
conditions $(BE1)$ and respectively $(BE3)$ hold. Now, the
compatibilities $(BE6)$ and $(BE7)$ are trivially fulfilled as
$k[G]$ and $k[X]$ are cocommutative coalgebras. Thus, it follows
from \thref{construnifbirp} that $(k[X], \, \lhd, \,
\rhd_{\gamma}, \,$ $f_{\gamma} )$ is a bialgebra extending
structure of $k[G]$ and we can construct the unified product $E :=
k[G] \ltimes k[X]$ associated to $(k[X], \, \lhd, \,
\rhd_{\gamma}, \, f_{\gamma} )$. As a vector space $ E = k[G] \ot
k[X]$ and has the multiplication given by \equref{10a} which in
this case takes the form:
$$
(g \ltimes x)\bullet(h \ltimes y) = g \, \gamma (x)\, h \,\gamma
(y) \, \gamma \Bigl ( (x \lhd h) \cdot y \Bigl)^{-1} \ltimes (x
\lhd h) \cdot y
$$
for all $g$, $h\in G$ and $x$, $y\in X$. The coalgebra structure
of $E$ is the tensor product of the group-like coalgebras $k[G]$
and $k[X]$.

The multiplication on the bialgebra $k[G] \ltimes k[X]$ can be simplified
as follows: using \prref{produsulciudat}, we obtain that there exists an isomorphism
of bialgebras
$$
k[G] \ltimes k[X] \cong k[G] \circledast k[X]
$$
where, the multiplication on $k[G] \circledast k[X]$ is given by
\equref{10b}:
$$
(g \circledast x)\bullet(h \circledast y) = gh \circledast \Bigl (
x \lhd \bigl( h \gamma (y) \bigl)^{-1} \Bigl) \, \cdot \, y
$$
for all $g$, $h\in G$ and $x$, $y\in X$.
\end{example}

\begin{example}
The construction of an explicit datum $(G, X, \cdot, \lhd,
\gamma)$ as in \exref{ultex} is a challenging problem considering
the compatibilities \equref{be1}-\equref{be3} that need to be
fulfilled by the map $\gamma : X \to G$. We give such an example
bellow. Let $G$ be a group, $(X, 1_{X})$ a pointed set such that
$(X, \lhd)$ is a right $G$-set with $1_X \triangleleft g = 1_X$,
for all $g \in G$. Let $ \cdot : X \times X \to X$ be the
following binary operation:
$$x \cdot y := \left \{\begin{array}{rcl}
x, \, & \mbox { if }& y = 1_X\\
y, \, & \mbox { if }& y \neq 1_X
\end{array} \right.
$$
Then, by a straightforward computation it can be seen that the
compatibility conditions \equref{be1} and \equref{be3} are
trivially fulfilled for any map $\gamma: X \to G$ such that
$\gamma(1_{X}) = 1_{G}$. Thus $(H := k[X], \, \lhd, \,
\rhd_{\gamma}, \, f_{\gamma} )$ is a bialgebra extending structure
of $k[G]$, where $\rhd_{\gamma}$, $f_{\gamma}$ are given by:
$$
x \rhd_{\gamma} g = \gamma (x) \, g \, \gamma (x \lhd g) ^{-1},
\quad f_{\gamma} (x, y) = \gamma (x) \gamma (y) \gamma (x \cdot
y)^{-1} = \left \{\begin{array}{rcl}
1_G, \, & \mbox { if }& y = 1_X\\
\gamma (x), \, & \mbox { if }& y \neq 1_X
\end{array} \right.
$$
for all $x$, $y\in X$ and $g\in G$.
\end{example}


\begin{thebibliography}{99}

\bibitem{am3}
Agore, A.L. and Militaru, G. - Extending structures II: the
quantum version, {\sl J. Algebra}, {\bf 336}(2011), 321--341.

\bibitem{AD}
Andruskiewitsch, N. and Devoto, J. - Extensions of Hopf algebras,
{\sl Algebra i Analiz} {\bf 7} (1995), 22--61.

\bibitem{AS}
Andruskiewitsch, N. and Schneider, H.-J. - On the classification
of finite-dimensional pointed Hopf algebras, {\sl Ann. Math.},
{\bf 171}(2010), 375–-417.

\bibitem{ABM}
Ardizzoni, A., Beatie, M. and Menini, C. - Cocycle deformations
for Hopf algebras with a coalgebra projection, {\sl J. Algebra}, {\bf 324}(2010), 673--705.

\bibitem{AM}
Ardizzoni, A., Menini, C. - Small Bialgebras with a Projection:
Applications, {\sl Comm. Algebra}, {\bf 37(8)} (2009), 2742--2784.

\bibitem{AMSt}
Ardizzoni, A., Menini, C. and Stefan, D. - A monoidal approach to
splitting morphisms of bialgebras, {\sl Trans. AMS}, {\bf 359}
(2007), 991�-1044

\bibitem{AMS}
Ardizzoni, A., Menini, C. and Stumbo, F. - Small Bialgebras with
Projection, J. Algebra, Vol. 314(2) (2007), 613-663.

\bibitem{molnar}
Molnar, R. K. - Semi-Direct Products of Hopf Algebras, {\sl J.
Algebra} {\bf 47} (1977), 29--51.

\bibitem{radford}
Radford, D. E. - The Structure of Hopf Algebras with a Projection,
{\sl J. Algebra} {\bf 92} (1985), 322--347.

\bibitem{Sch2}
Schauenburg, P. - The structure of Hopf algebras with a weak
projection, {\sl Algebr. Represent. Theory} {\bf 3} (1999), 187
--211.
\end{thebibliography}
\end{document}